\documentclass[11pt]{amsart}

\usepackage{amsmath,amssymb,latexsym,soul,cite,mathrsfs}
\usepackage{color}
\usepackage{color,enumitem,graphicx}
\usepackage[colorlinks=true,urlcolor=blue,
citecolor=red,linkcolor=red,linktocpage,pdfpagelabels,
bookmarksnumbered,bookmarksopen]{hyperref}
\usepackage[english]{babel}

\usepackage[left=2.9cm,right=2.9cm,top=2.8cm,bottom=2.8cm]{geometry}
\usepackage[hyperpageref]{backref}

\usepackage[colorinlistoftodos]{todonotes}
\makeatletter
\providecommand\@dotsep{5}
\def\listtodoname{List of Todos}
\def\listoftodos{\@starttoc{tdo}\listtodoname}
\makeatother

\numberwithin{equation}{section}

\newcommand{\Om} {\Omega}
\newcommand{\la} {\lambda}

\newcommand{\mb} {\mathbb}
\newcommand{\mc} {\mathcal}
\newcommand{\noi} {\noindent}

\newtheorem{theorem}{Theorem}[section]

\newtheorem{lemma}[theorem]{Lemma}

\newtheorem{remark}{Remark}
\newtheorem{definition}{Definition}[section]

\begin{document}

\title[Existence and multiplicity of solutions  ...]
{The Nehari manifold method for Fractional Kirchhoff problem involving singular and exponential nonlinearity}
\author{ Tuhina Mukherjee and Mingqi Xiang}

\address[Tuhina Mukherjee ]
{\newline\indent Department of Mathematics,
	\newline\indent
National Institute of Technology Warangal
	\newline\indent
Hanamkonda, Warangal-506004, India}

\pretolerance10000


\begin{abstract}
\noindent In this paper we establish the existence of at least two weak solutions for the following fractional Kirchhoff problem involving singular and exponential nonlinearity
\begin{equation*}
   \left\{\begin{split}
        M\left(\|u\|^{\frac{n}{s}}\right)(-\Delta)^s_{n/s}u & = \mu u^{-q}+ u^{r-1}\exp( u^{\beta})\;\text{in}\;\Om,\\
        u&>0,\;\text{in}\; \Om,\\
        u &= 0,\;\text{in}\; \mb R^n \setminus{\Om},
    \end{split}
    \right.
\end{equation*}
where $\Om$ is smooth bounded domain in $\mb R^n$, {$n\geq 1$}, $s\in (0,1)$, $\mu>0$ is a real parameter, $\beta <\frac{n}{n-s}$ and $q\in (0,1)$.
We have considered the degenerate Kirchhoff case here and used the Nehari manifold techniques to obtain the results. 
\end{abstract}


\maketitle

\section{Introduction}
We are interested to study the multiple solutions of the following doubly nonlocal problem which involves a singular term and {subcritical} nonlinearity of exponential type
\begin{equation*}
   (P_{\mu}) \left\{\begin{split}
        M\left(\|u\|^{\frac{n}{s}}\right)(-\Delta)^s_{n/s}u & = \mu u^{-q}+ u^{r-1}\exp( u^{\beta})\;\text{in}\;\Om,\\
        u&>0,\;\text{in}\; \Om,\\
        u &= 0,\;\text{in}\; \mb R^n \setminus{\Om},
    \end{split}
    \right.
\end{equation*}
where $\Om$ is smooth bounded domain in $\mb R^n$, {$n\geq 1$}, $s\in (0,1)$, $\mu>0$ is a real parameter, $\beta <\frac{n}{n-s}$ and $q\in (0,1)$. We consider the degenerate Kirchhoff case with subcritical nonlinearity 
where we precisely assume that the function $M: \mb R^+ \to \mb R^+$ is defined as\[M(t)= bt^{\theta -1},\; \text{for} \; t>0,\]
where   $\theta >1$ such that $r> \frac{\theta n}{s}$. The  $n/s$-fractional Laplace operator (up to normalizing factor) $(-\Delta)_{n/s}^s$ is defined as 
\[(-\Delta)_{n/s}^su(x) := \lim_{\epsilon\to 0}\int_{\mb R^n\setminus B_\epsilon(x)} \frac{|u(x)-u(y)|^{\frac{n}{s}-2}(u(x)-u(y))}{|x-y|^{2n}}~dy,\;\;\text{for}\;\; x \in \mb R^n\]
where $u\in C_c^\infty(\mb R^n)$ and $B_\epsilon(x)$ denotes ball in $\mb R^n$ with centre $x$ and radius $\epsilon$. In literature, Kirchhoff problems has been divided into two cases- \textit{degenerate case} when $M(0)=0$ and \textit{non degenerate case} when $M(t)\geq a>0$ for some constant $a$, for all $t\in \mb R^+$.
The following singular Kirchhoff problem involving fractional Laplacian and Sobolev critical nonlinearity
\begin{equation}\label{KSE-lit-1}
  \begin{split}
        M\left(\int\int_{\mb R^{2n}}\frac{|u(x)-u(y)|^{2}}{|x-y|^{n+2s}}~dxdy\right)(-\Delta)^s u & = \la u^{-q}+ g(x)u^{2^*_s-1}\;\text{in}\;\Om,\\
        u&>0,\;\text{in}\; \Om,\\
        u &= 0,\;\text{in}\; \mb R^n \setminus{\Om},
    \end{split}
\end{equation}
where $n>2s$, $q\in (0,1)$ and $2^*_s=\frac{2n}{n-2s}$ has been recently studied in \cite{fisc} and \cite{fisc-pawan}. In \cite{fisc}, A. Fiscella proved that \eqref{KSE-lit-1} admits at least two distinct solutions in the degenerate case when $g\equiv 1$ and $\la$ has a certain specific range using truncation method along with Mountain Pass theorem. Whereas, authors in \cite{fisc-pawan} used the Nehari manifold technique to establish existence of at least two distinct solutions to \eqref{KSE-lit-1} in the non degenerate case when $g \in L^\infty(\Om)$ is sign-changing function. We also refer a very recent article \cite{WCZ}. Motivated by these articles, we studied $(P_\mu)$ to prove that it has at least two weak solutions and the novelty of the article lies in the fact that the Kirchhoff singular problem with exponential nonlinearity has not even been studied in the Laplacian case. We has used the Nehari manifold technique to prove our results. So the article deals with a complete open case of singular Kirchhoff problems with exponential nonlinearity and contains new estimates.

We now present some references to articles in literature dealing with singular problems for the readers.  Consider the following $p$-Laplace equation
\begin{equation}\label{plap}
-\Delta_p u=\la\frac{f(x)}{u^{q}}+\mu u^{r}\; \text{in}\;\Om;\\
~~u>0\;\text{in}\; \Om;\\
~~u=0\;\text{on}\; \partial\Om.
\end{equation}
For $p=2$, $\mu=0$ and $f(x)\equiv 1$, existence of a classical solution to the problem \eqref{plap} is proved in \cite{CRT} for any $q>0.$ Later for certain restricted range of $q$, the existence of weak solution was proved in \cite{LMckena}. {This restriction on $q$} was removed in \cite{Boc3} to obtain existence of at least one weak solution. Indeed the authors in \cite{Boc3} proved the existence of solution in $H_{0}^1(\Om)$ for $0<q\leq 1$ and in $H^{1}_{loc}(\Om)$ for $q>1$. The case of $p\in(1,\infty)$ was settled in \cite{Canino1}, where existence of weak solution in $W_{0}^{1,p}(\Om)$ was proved for $0<q\leq 1$ and in $W^{1,p}_{loc}(\Om)$ for $q>1.$

On the other hand, for $p=2,$ $f(x)\equiv \mu=1$ and $1<r\leq 2^*= \frac{2N}{N-2}$, multiplicity of weak solutions was established using Nehari manifold and sub super solution techniques in \cite{YHaitao,hirano1} for $0<q<1.$ Whereas the case of any $q>0$ was settled in \cite{DArcoya, hirano2}. In the nonlinear case that is for $p\in(1,\infty),$ authors in \cite{GST} answered the question of existence, multiplicity and regularity of weak solutions in the case of $0<q<1$ which was further extended to the case of $q\geq 1$ in \cite{BalGarain} deducing the  multiplicity of weak solutions. In context of singular problems involving fractional Laplace operator, authors in \cite{peral-frac} studied the following problem
\begin{equation*}
(-\Delta)^s u = \la\frac{f(x)}{u^\gamma} + M u^{p}, \;
 u>0\;\text{in}\;\Om, \quad
 u = 0 \; \mbox{in}\; \mb R^N \setminus\Om,
\end{equation*}
where  $N>2s$, $M\geq 0$, $0<s<1$, $\gamma,\;\la>0$, $1<p<2_{s}^{*}-1,\; 2^*_s= 2N/(N-2s)$, $f\in L^{m}(\Om)$ for $m\geq 1$ is a non-negative function and
\[ (-\Delta)^s u(x) = - \frac{1}{2}\int_{\mb R^N} \frac{u(x+y)+u(x-y)-2u(x)}{|y|^{N+2s}}~dy , \; \; \text{for all }\; x \in \mb R^N.\] Here authors studied the existence of distributional solutions for small $\la$ using the uniform estimates of  $\{u_n\}$ which are solutions of the regularized problems with singular  term $u^{-\gamma}$ replaced by $(u+\frac{1}{n})^{-\gamma}$. This was extended for the $p$-fractional Laplace operator by Canino et al. in \cite{Canino3}. In the critical case for $0<q<1$, the question of existence and multiplicity of weak solutions to nonlocal singular problems has been answered in \cite{mukh-tmna, mukh-ejde, mukh-anona} whereas $q\geq 1$ case has been dealt in \cite{mukh-anona1}. We also refer a recent article \cite{divya} related to nonlocal singular problem with exponential nonlinearity. The Kirchhoff singular problems involving Laplace operator has been investigated in \cite{llt,lzlt}.\\

\noi This article is divided into five sections. Section $2$ contains the preliminaries for the article and its main result. The heart of Nehari manifold technique that is the fibering map analysis is carried out in Section $3$. Finally the existence of first and second solution to $(P_\mu)$ has been established in Section $4$ and $5$ respectively. 
\section{Preliminaries}
\noi Let us consider the usual fractional Sobolev space
\[W^{s,p}(\Om):= \left\{u\in L^{p}(\Om); \frac{(u(x)-u(y))}{|x-y|^{\frac{n}{p}+s}}\in L^{p}(\Om\times\Om)\right\}\]
 endowed with the norm
\begin{align*}
\|u\|_{W^{s,p}(\Om)}=\|u\|_{L^p(\Om)}+ \left(\int_{\Om}\int_{\Om}
\frac{|u(x)-u(y)|^{p}}{|x-y|^{n+ps}}dxdy \right)^{\frac 1p}
\end{align*}
where $\Om \subset \mb R^n$ is an open set. {We denote $W^{s,p}_0(\Omega)$ as the completion of the space $C_c^\infty(\Omega)$ with respect to the norm $\|\cdot\|_{W^{s,p}(\Omega)}$.}
Now we define
 \[ X_0 = \{u\in W^{s,n/s}(\mb R^n) : u = 0 \;\text{in}\; \mb R^n\setminus \Om\}\]
with respect to the norm
\[\|u\|_{X _0}=\left( \int_{\mb R^{n}}\int_{\mb R^n}\frac{|u(x)- u(y)|^{\frac{n}{s}}}{|x-y|^{2n}}dx
dy\right)^{\frac sn}= \left(\int_{Q}\frac{|u(x)- u(y)|^{\frac{n}{s}}}{|x-y|^{2n}}dx
dy\right)^{\frac sn},\]
where  $Q=\mb R^{2n}\setminus(\mc C\Om\times \mc C\Om)$ and
 $\mc C\Om := \mb R^n\setminus\Om$. Then $X_0$ is a reflexive Banach space and continuously embedded in $W^{s,p}_0(\Om)$. Also $X_0 \hookrightarrow \hookrightarrow L^q(\Om)$ compactly for each $q \in [1,\infty)$. Note that the norm
$\|.\|_{X_0}$ involves the interaction between $\Om$ and $\mb
R^n\setminus\Om$. We denote $\|.\|_{X_0}$ by $\|.\|$ in future, for notational convenience in future.

The study of elliptic equations involving nonlinearity with exponential growth are motivated by
the following Trudinger-Moser inequality in \cite{martinazi}, namely
\begin{theorem} \label{moser}
Let $\Om$ be a open bounded domain then we define $\tilde{W}^{s,n/s}_{0}(\Om)$ as the completion of $C_{c}^{\infty}(\Om)$ with respect to the norm $\|u\|^{\frac{n}{s}}= \displaystyle\int_{\mb R^{n}}\int_{\mb R^n} \frac{|u(x)-u(y)|^{\frac{n}{s}}}{|x-y|^{2n}} dx dy$. Then there exists a positive constant $\alpha_{n,s}$ given by
\[\alpha_{n,s}= \frac{n}{\omega_{n-1}}\left( \frac{\Gamma(\frac{n-s}{2})}{\Gamma(s/2)2^s \pi^{n/2}}\right)^{-\frac{n}{n-s}},\]
where $\omega_{n-1}$ be the surface area of the unit sphere in $\mb R^n$ and $C_{n,s}$ depending only on $n$ and $s$ such that
\begin{equation}\label{TM-ineq}
    \sup_{u \in \tilde{W}^{s,n/s}_{0}(\Om),\; \|u\|\leq 1}\int_\Om \exp\left( \alpha |u|^{\frac{n}{n-s}}\right)~dx\leq C_{n,s}|\Om|
    \end{equation}
for each $\alpha \in [0,\alpha_{n,s}]$.
Moreover there exists a $ \alpha_{n,s}^* \geq \alpha_{n,s}$ such that {the right hand side of\eqref{TM-ineq} is $+\infty$} for $\alpha>\alpha_{n,s}^*$.
\end{theorem}
It is proved in \cite{ruf} (see Proposition 5.2) that
\[\alpha_{n,s}^*= n \left(\frac{2(n\mc W_n)^2 \Gamma(\frac{n}{s}+1)}{n!}\sum_{i=0}^{\infty} \frac{(n+i-1)!}{i! (n+2i)^{\frac{n}{s}}}\right)^{\frac{s}{n-s}},\]
where $\mc W_n= \frac{w_{n-1}}{n}$ {is} the volume of the unit sphere in $\mb R^n$. It {is still} unknown whether $ \alpha_{n,s}^* = \alpha_{n,s}$ {or not}.
\begin{definition}
We call a $u \in X_0$ as a weak solution to $(P_{\mu})$ if for any $\varphi\in X_0$,  $u^{-q}\varphi \in L^1(\Om)$ and it satisfies
\begin{align*}
&(\|u\|^{\frac{n\theta}{s}})\int_{\mb R^{2n}}\frac{|u(x)-u(y)|^{\frac{n}{s}-2} (u(x)-u(y))(\varphi(x)-\varphi(y))}{|x-y|^{2n}}~dxdy\\
\quad &= \mu\int_{\Om}(u^+)^{-q}\varphi~dx + \int_{\Om} (u^+)^{r-1}\exp((u^+)^{\beta})\varphi~dx
\end{align*}
\end{definition}
Keeping in mind the requirement of positive solution and presence of singular term in $(P_{\mu})$, we have considered '$u^+$' instead of '$u$' to define our weak solution.
Let us set $f(z)=  z^{r-1}\exp( z^{\beta})$ for $z\geq 0$ and $F(z)= \int_0^z f(t)~dt$.

It is obvious to notice that $(P_\mu)$ has a variatonal characterization so that the corresponding energy functional $I_{\mu}: X_0\to \mb R$ is given by
\[I_\mu(u) := \frac{s}{n\theta}\|u\|^{\frac{n\theta}{s}} -\frac{\mu}{1-q}\int_{\Om}(u^+)^{1-q}~dx- \int_{\Om}F(u^+)~dx.\]
Certainly because of presence of singular term, $I_\mu$ fails to be Frech\'{e}t differentiable but $q\in (0,1)$ certainly makes it a continuous functional on $X_0$. So we can not rely upon the critical point theory to prove the existence results for $(P_\mu)$ and hence appropriately use the fibering maps and the Nehari manifold set to carry out the research.

We state our main result.
\begin{theorem}\label{KSE-MT}
There exists a $\mu_0>0$ such that whenever $\mu \in (0,\mu_0)$, $(P_\mu)$ admits at least two non negative solutions in $X_0$.
\end{theorem}

\section{Fibering Map Analysis}
\noi For fixed $u \in X_0$, we define the fiber map $\phi_u: \mb R^+ \to \mb R $ as $\phi_u(t)= I_\mu(tu)$ for $t>0$. Precisely,
\[\phi_u (t) =  \frac{s}{n\theta}\|tu\|^{\frac{n\theta}{s}} -\frac{t^{1-q}\mu}{1-q}\int_{\Om}(u^+)^{1-q}~dx- \int_{\Om}F(tu^+)~dx\]
so that we can easily derive
\[\phi_u^\prime(t) =  t^{\frac{n\theta}{s}-1}\|u\|^{\frac{n\theta}{s}}-\mu t^{-q}\int_{\Om}(u^+)^{1-q}~dx-\int_{\Om}f(tu^+)u^+~dx\]
and
\[\phi_u^{\prime \prime}(t) = \left(\frac{n\theta}{s}-1\right) t^{\frac{n\theta}{s}-2}\|u\|^{\frac{n\theta}{s}} +q\mu t^{-q-1}\int_{\Om}(u^+)^{1-q}~dx-\int_{\Om}f^\prime(tu^+)(u^+)^2~dx.\]
Now we set our Nehari manifold as follows
\[\mc N_{\mu}=\{u\in X_0:\; \phi_u^\prime(1)=0\}\]
which means that every weak solution of $(P_{\mu})$ naturally belongs to $\mc N_\mu$. Depending on the nature of critical points of $\phi_u$, we subdivide $\mc N_\mu$ into following three sets-
\[\mc N_{\mu}^{\pm}=\{u\in \mc N_{\mu}:\; \pm\phi_u^{\prime \prime}(1)>0\}\; \text{and}\; \mc N_\mu:= \{u \in \mc N_{\mu}:\; \phi_u^{\prime \prime}(1)=0\}.\]
It is easy to see that $tu \in \mc N_\la$ if and only if $\phi^\prime_u(t)=0$. Since $\phi_u^\prime (t)=0$ iff
\[m_u(t)= \mu\int_{\Om}(u^+)^{1-q}~dx,\]
where
\[m_u(t) :=  t^{\frac{n\theta}{s}+q-1}\|u\|^{\frac{n\theta}{s}}-\la t^{q}\int_{\Om}f(tu^+)u^+~dx, \; \text{for}\; t>0,\]
we will study the map $m_u(t)$ now. It is easy that $m_u(0)=0$ and since $\frac{n\theta}{s}+q-1< q+r-1$ so $t^{q+r-1}< t^{\frac{n\theta}{s}+q-1}$ and $f$ has exponential growth which implies that $m_u(t)>0$ for small enough $t$. Moreover,
\[\lim_{t\to +\infty}m_u(t)=-\infty.\]
Thus there exists at least one $t^*>0$ such that it is point of local maximum and satisfies
\begin{equation}\label{KSE-60}
m_u^\prime(t^*)=0
\end{equation}
which gives us that
\begin{equation}\label{KSE-1}
\|t^*u\|^{\frac{n\theta}{s}}=\frac{\left(\displaystyle q\int_{\Om}f(t^*u^+)t^*u^+~dx +\int_\Om f^\prime(t^*u^+)(t^*u^+)^2~dx\right)}{\left(\frac{n\theta}{s}+q-1\right)}
\end{equation}

\begin{lemma}\label{KSE-lemma1}
Let
\[\Gamma:= \left\{u\in X_0:\; {\|u\|^{\frac{n\theta}{s}} \leq \frac{\left(\displaystyle q\int_{\Om}f(u^+)u^+~dx +\int_\Om f^\prime(u^+)(u^+)^2~dx\right)}{\left(\frac{n\theta}{s}+q-1\right)}}\right\}\]
and
\[\Gamma_0:= \inf_{u\in\Gamma\setminus \{0\}} \left\{ \left(1-\frac{n\theta}{s}\right)\int_{\Om}f(u^+)u^+~dx+ f^\prime(u^+)(u^+)^2~dx-\mu \left(\frac{n\theta}{s}+q-1\right)\int_\Om (u^+)^{1-q}~dx\right\}.\]
Then there exists a $\mu_0>0$ such that $\Gamma_0>0$ when $\mu \in (0,\mu_0)$.
\end{lemma}
\begin{proof}
First we claim that $\inf\limits_{u\in\Gamma\setminus \{0\}}\|u\|>0$. Suppose this is not true then there exists a sequence $\{u_k\}\subset \Gamma\setminus \{0\}$ such that $\|u_k\|\to 0$ which implies that
{\begin{equation}\label{KSE-2}
    \|u_k\|^{\frac{n\theta}{s}} \leq \frac{\left(\displaystyle q\int_{\Om}f(u_k^+)u_k^+~dx +\int_\Om f^\prime(u_k^+)(u_k^+)^2~dx\right)}{\left(\frac{n\theta}{s}+q-1\right).}
\end{equation}
Since $\{u_k\}$ converges to 0, we can assume that
$\|u_k^+\|^\beta<\alpha<\alpha_{n,s}$ for $k$ large.
Choosing $\mu>1$ close to 1 such that $\mu \alpha<\alpha_{n,s}$ we get $\mu\|u_k^+\|^\beta<\alpha_{n,s}$.} So consider
\begin{align*}
    &q\int_{\Om}f(u_k^+)u_k^+~dx +\int_\Om f^\prime(u_k^+)(u_k^+)^2~dx \\
    & =\int_\Om f(u_k^+)u_k^+ \left(r+q-1+\beta (u_k^+)^{\beta}\right)~dx\\
    &  =\int_\Om (u_k^+)^{r}\exp\left((u_k^+)^{\beta}\right) \left(r+q-1+\beta(u_k^+)^{\beta}\right)~dx\\
    &= \int_\Om (u_k^+)^{r}\exp\left(\|u_k^+\|^{\beta}\left(\frac{u_k^+}{\|u_k^+\|}\right)^{\beta}\right) \left(r+q-1+\beta(u_k^+)^{\beta}\right)~dx\\
    & \leq \left(\int_\Om\exp\left(\mu\|u_k^+\|^{\beta}\left(\frac{u_k^+}{\|u_k^+\|}\right)^{\beta}\right)dx\right)^{\frac{1}{\mu}}
    \left(\int_\Om |u_k|^{\mu^\prime r}\left(r+q-1+\beta|u_k|^{\beta}\right)^{\mu^\prime}dx\right)^{\frac{1}{\mu^\prime}}\\
    &\leq C (\|u_k\|^r + \|u_k\|^{r+\beta}),
\end{align*}
where we used $u_k^+ \leq |u_k|$
and $\mu^\prime$ denotes H\"{o}lder conjugate of $\mu$.

Then we use Theorem \ref{moser} to obtain the constant $C>0$ which depends on $n,\theta,q,r,\alpha_0$. Using this in \eqref{KSE-2} we get that
\[1\leq C \left(\|u_k\|^{r-\frac{n\theta}{s}} + \|u_k\|^{r+\beta-\frac{n\theta}{s}}\right)\]
which gives a contradiction if $\|u_k\|\to 0$, since $r>\frac{n\theta}{s}$.

Now consider
\begin{align}
&\left(1-\frac{n\theta}{s}\right)\int_{\Om}f(u^+)u^+~dx+ f^\prime(u^+)(u^+)^2~dx-\mu \left(\frac{n\theta}{s}+q-1\right)\int_\Om (u^+)^{1-q}~dx\nonumber\\
&= \int_\Om (u^+)^{r}\exp\left((u^+)^{\beta}\right) \left(r-\frac{n\theta}{s}+\beta(u^+)^{\beta}\right)~dx -\mu \left(\frac{n\theta}{s}+q-1\right)\int_\Om (u^+)^{1-q}~dx\label{KSE-3}
\end{align}
where we can estimate the second term as follows by fixing $\eta = \frac{r+ \beta}{1-q}>1$ and using H\"{o}lder inequality
\begin{align*}
    \int_\Om (u^+)^{1-q}~dx &\leq \int_\Om |u|^{1-q}~dx \\
    & \leq \left(\int_\Om |u|^{\eta(1-q)}~dx\right)^{\frac{1}{\eta}}
    \left(|\Om|\right)^{\frac{1}{\eta^\prime}}= \left(\int_\Om |u|^{r+\beta}~dx\right)^{\frac{1}{\eta}}
    \left(|\Om|\right)^{\frac{1}{\eta^\prime}}\\
    & \leq \left(\int_\Om |u|^{r}\exp\left(|u|^{\beta}\right) \left(r-\frac{n\theta}{s}+\beta|u|^{\beta}\right)~dx\right)^{\frac{1}{\eta}}
    \left(|\Om|\right)^{\frac{1}{\eta^\prime}}.
\end{align*}
Therefore using $2\leq \frac{n\theta}{s}$, \eqref{KSE-3} becomes
\begin{equation}\label{KSE-5}
\begin{split}
    &\left(1-\frac{n\theta}{s}\right)\int_{\Om}f(u^+)u^+~dx+ f^\prime(u^+)(u^+)^2~dx-\mu \left(\frac{n\theta}{s}+q-1\right)\int_\Om (u^+)^{1-q}~dx\\
&\geq  \int_\Om |u|^{r}\exp\left(|u|^{\beta}\right) \left(r-\frac{n\theta}{s}+\beta|u|^{\beta}\right)~dx \\
&\quad \quad-\mu \left(\frac{n\theta}{s}+q-1\right)\left(\int_\Om |u|^{r}\exp\left(|u|^{\beta}\right) \left(r-\frac{n\theta}{s}+\beta|u|^{\frac{n}{n-s}}\right)~dx\right)^{\frac{1}{\eta}}
    \left(|\Om|\right)^{\frac{1}{\eta^\prime}}.
    \end{split}
\end{equation}
Our second claim is
\[K:=\inf_{u\in \Gamma\setminus\{0\}}\int_\Om |u|^{r}\exp\left(|u|^{\beta}\right) \left(r-\frac{n\theta}{s}+\beta|u|^{\beta}\right)~dx>0. \]
which easily follows from our first claim and \eqref{KSE-2} which implies $0<\inf_{u \in \Gamma\setminus\{0\}} f^\prime(|u|)|u|^2$. With this, \eqref{KSE-5} can be rewritten as
\begin{equation}\label{KSE-6}
\begin{split}
    &\left(1-\frac{n\theta}{s}\right)\int_{\Om}f(u^+)u^+~dx+ f^\prime(u^+)(u^+)^2~dx-\mu \left(\frac{n\theta}{s}+q-1\right)\int_\Om (u^+)^{1-q}~dx\\
&\geq  \int_\Om |u|^{r}\exp\left(|u|^{\beta}\right) \left(r-1+\beta|u|^{\beta}\right)\left[
1-\mu \left(\frac{n\theta}{s}+q-1\right)
    \left(\frac{|\Om|}{K}\right)^{\frac{1}{\eta^\prime}}\right]dx.
    \end{split}
\end{equation}
Hence assuming
\[0<\mu < \mu_0:= \frac{\left(\frac{K}{|\Om|}\right)^{\frac{1}{\eta^\prime}}}{\left(\frac{n\theta}{s}+q-1\right)}\]
completes the proof.
\end{proof}
The following result illustrates that the decomposition of $\mc N_{\mu}$ are in fact nonempty.

\begin{theorem}\label{KSE-theorem1}
For any $u\in X_0\setminus \{0\}$, there exists a unique $t^*=t^*(u)>0$, $t^+=t^+(u)$ and $t^-=t^-(u)$ satisfying $t^+<t^*<t^-$ such that $t^+u\in \mc N_\mu^+$ and $t^-u\in \mc N_\mu^-$ whenever $\mu \in (0,\mu_0).$ Moreover, $I_\mu(t^+u)= \min_{t\in[0,t^-]}I_\mu(tu)$ and $I_\mu^\prime(t^-u)=\max_{t\geq t^*}I_\mu(tu)$.
\end{theorem}
\begin{proof}
From \eqref{KSE-60}, we have the existence of a $t^*>0$ such that $m_u^\prime(t^*)=0$ and it is a point of local maximum. Using the definition of $m_u(t)$ and \eqref{KSE-1} we get
\begin{align*}
    m_u(t^*) &= (t^*)^{\frac{n\theta}{s}+q-1}\|u\|^{\frac{n\theta}{s}}-\la (t^*)^{q}\int_{\Om}f(t^*u^+)u^+~dx \\
    & = \frac{(t^*)^{q-1}}{\left(\frac{n\theta}{s}+q-1\right)}\left(q\int_\Om f(t^*u)t^*u~dx+ \int_\Om f^\prime(t^*u)(t^*u)^2~dx\right)-(t^*)^q\int_{\Om}f(t^*u)u~dx\\
    &= \frac{(t^*)^{q-1}}{\left(\frac{n\theta}{s}+q-1\right)}\left[ \left(1-\frac{n\theta}{s}\right)\int_\Om f(t^*u)(t^*u)~dx + \int_\Om f^\prime(t^*u)(t^*u)^2~dx\right].
\end{align*}
Since \eqref{KSE-1} holds, we get that $t^*u\in \Gamma\setminus \{0\}$. This gives
\begin{align*}
    m_u(t^*)- \mu\int_\Om (u^+)^{1-q}~dx
   &=  \frac{(t^*)^{q-1}}{\left(\frac{n\theta}{s}+q-1\right)}\left[ \left(1-\frac{n\theta}{s}\right)\int_\Om f(t^*u)(t^*u)~dx + \int_\Om f^\prime(t^*u)(t^*u)^2~dx\right.\\
   &\quad\quad\left.-\mu\left(\frac{n\theta}{s}+q-1\right)\int_\Om (t^*u^+)^{1-q}~dx\right]\\
   &\geq \frac{(t^*)^{q-1}}{\left(\frac{n\theta}{s}+q-1\right)}\Gamma_0>0
\end{align*}
using Lemma \ref{KSE-lemma1}. Hence there exist $t^+$ and $t^-$ both depending on $u$ such that $0<t^+<t^*<t^-$ and $t^+u,\;t^-u \in \mc N_\mu$. We claim that $t^*$ is actually unique critical point of $m_u(t)$ which is in fact the global maximum. Because if it is not so, then there exists another $t^{**}>0$ such that ${m_u^\prime(t^{**})=0}$ and it is a point of local minimum. Also there exists $t^{++}$ and $t^{--}$ both depending on $u$ such that $0<t^{++}<t^{**}<t^{--}$ and  $t^{++}u,\;t^{--}u \in \mc N_\mu$. But by virtue of Lemma \ref{KSE-lemma1}, we also have
\[m_u(t^{**})> \mu \int_\Om (u^+)^{1-q}~dx\]
which contradicts the fact that $t^{**}$ is a point of local minimum and existence of $t^{++}$ and $t^{--}$. This proves the claim. Now since $t^*$ is a point of global maximum of $m_u(t)$, we infer that $m_u(t)$ is increasing in $(0,t^*)$ and decreasing in $(t^*,\infty)$.
This along with the relation
\[\phi_{tu}^{\prime\prime}(1)= t^{2-q}m_u^\prime(t), \; \text{for}\; tu \in \mc N_{\mu}\]
asserts that $t^+u\in \mc N_\mu^+$ and $t^-u \in \mc N_\mu^-$. Lastly considering
\[\phi_u^\prime(t) = t^{-q}\left(m_u(t)-\int_{\Om}(u^+)^{1-q}~dx\right),\; \text{for}\; t>0\]
gives that $\phi_u^\prime(t)<0$ for $t \in [0,t^+)$ and $\phi_u^\prime(t)>0$ for $t \in (t^+,t^-)$. Thus
\[I_\mu(t^+u)= \min_{t\in[0,t^-]}I_\mu(tu).\]
On a similar note, $\phi_u^\prime(t)>0$ when $t\in(t^+,t^-)$, $\phi_u^\prime(t^-)=0$ and $\phi^\prime_u(t)<0$ when $t\in (t^-,\infty)$ yields that
\[I_\mu^\prime(t^-u)=\max_{t\geq t^*}I_\mu(tu)\]
which completes the proof.
\end{proof}
As a consequence of above results, we state our next important result.
\begin{lemma}\label{KSE-lemma2}
The set $\mc N_\mu^0=\{0\}$ when $\mu \in (0,\mu_0)$.
\end{lemma}
\begin{proof}
We prove it by contradiction, so assume that there exists a $u\in \mc N_\mu^0$ such that $u \not\equiv 0$. Then $\phi_u^\prime(1)=0=\phi_u^{\prime\prime}(1)$ gives us
\begin{align}
    \|u\|^{\frac{n\theta}{s}} &= \mu\int_\Om (u^+)^{1-q}~dx + \int_\Om f(u^+)u^+~dx\label{KSE-7}\\
    \left(\frac{n\theta}{s}-1\right)\|u\|^{\frac{n\theta}{s}} &= -\mu q \int_\Om (u^+)^{1-q}~dx + \int_\Om f^\prime(u^+)(u^+)^2~dx.\label{KSE-8}
\end{align}
Multiplying \eqref{KSE-7} by $q$ and adding it to \eqref{KSE-8} gives
\[\|u\|^{\frac{n\theta}{s}}=\frac{\displaystyle q\int_{\Om}f(u^+)u^+~dx+ \int_{\Om}f^\prime(u^+)(u^+)^2~dx}{\left(\frac{n\theta}{s}+q-1\right)}\]
which implies that $u \in \Gamma \setminus \{0\}$. Putting the value of $\|u\|^{\frac{n\theta}{s}}$ from \eqref{KSE-7} to \eqref{KSE-8}, we get
\[\mu\left(\frac{n\theta}{s}+q-1\right)\int_\Om(u^+)^{1-q}~dx=\int_\Om\left(1-\frac{n\theta}{s}\right)\int_\Om f(u^+)u^+~dx + \int_\Om f^\prime(u^+)(u^+)^2~dx\]
contradicting that $\Gamma_0>0$ since $\mu \in (0,\mu_0)$. This finishes the proof.
\end{proof}

Although $I_\mu$ fails to be bounded over whole $X_0$, in the preceding Lemma we prove that $I_\mu$ is bounded below on $\mc N_\mu$.

\begin{lemma}\label{KSE-lemma3}
The energy functional $I_\mu$ is coercive and bounded below on $\mc N_\mu$.
\end{lemma}
\begin{proof}
Let $u \in \mc N_\mu$. It is easy to see that $rF(t)\leq f(t)t$ for $t\geq 0$. Using this along with Sobolev embedding, we get
\begin{align*}
    I_\mu(u) &= I_\mu(u) -\frac{1}{r}\phi_u^\prime(1)\\
    & =\left(\frac{s}{n\theta}-\frac{1}{r} \right)\|u\|^{\frac{n\theta}{s}} -\mu \left(\frac{1}{1-q}-\frac{1}{r}\right)\int_{\Om}(u^+)^{1-q}~dx+ \left(\frac{1}{r}\int_\Om f(u^+)u^+~dx - \int_\Om F(u^+)~dx\right)\\
    & \geq \frac{\left(r-\frac{n\theta}{s}\right)}{\frac{rn\theta}{s}}\|u\|^{\frac{n\theta}{s}} -\mu C \left(\frac{r-1+q}{r(1-q)}\right)\|u\|^{1-q}
\end{align*}
where $C>0$ is constant. Since $1-q<2 \leq \frac{n\theta}{s}$, the above estimate implies that $I_\mu$ is necessarily coercive on $\mc N_\mu$. Furthermore, if we set
\[g(l):= K_1l^{\frac{n\theta}{s}}-K_2l^{1-q} \]
where \[K_1 = \frac{\left(r-\frac{n\theta}{s}\right)}{\frac{rn\theta}{s}}\; \text{and}\; K_2 = \mu C \left(\frac{r-1+q}{r(1-q)}\right)\]
then it is easy to verify that $g$ has a unique critical point which is a point of global minimum to it, given by
\[l_{min} = \left(\frac{(1-q)K_2}{K_1}\right)^{\frac{1}{\frac{n\theta}{s}-1+q} }.\]
Hence we get $I_\mu(u)\geq g(l_{min})$ which implies that $I_\mu$ is bounded below on $\mc N_\mu$.
\end{proof}

\section{Existence of first solution}
 In this section, we establish the existence of first solution to $(P_\mu)$ by solving a minimization problem over $\mc N_\mu^+$.

 \begin{lemma}\label{KSE-lemma4}
 For every $u\in \mc N_\mu^+$ and $\mu \in (0,\mu_0)$, there exists a positive real number $\epsilon$ and a differentiable function $\xi : B(0,\epsilon)\subset X_0 \to \mb R^+$ such that
 \[\xi(v)>0,\; \xi(0)=1,\; \xi(v)(u-v)\in \mc N_\mu^+,\;\text{for all}\; v \in B(0,\epsilon).\]
 \end{lemma}
 \begin{proof}
 We consider $u\in \mc N_\mu$ and define the map $G_u: \mb R \times X_0 \to \mb R$ as
 \[G_u(t,v)= t^{\frac{n\theta}{s}+q-1}\|u-v\|^{\frac{n\theta}{s}}-\mu \int_{\Om}((u-v)^+)^{1-q}~dx-t^{q+r-1}\int_\Om f(t(u-v)^+)(u-v)^+~dx.\]
 This implies that $G_u(1,0)= \phi_u^\prime(1)=0$ since $u\in \mc N_\mu$ and with some computation, we get
 \[\frac{\partial }{\partial t}G_u(1,0)=\phi^{\prime \prime}(1)>0\]
 since $u\in \mc N_\mu^+$. Now applying the Implicit Function Theorem at $(1,0)$, we get that there exists a $\epsilon>0$ and a differentiable function $\xi : B(0,\epsilon)\subset X_0 \to \mb R^+$ such that $\xi(0)=1$ and $G_u(\xi(v),v)=0$ when $v \in B(0,\epsilon)$. Since $G_u(\xi(v),v)= (\xi(v))^q\phi^\prime_{u-v}(\xi(v))$, this implies that $\phi^\prime_{u-v}(\xi(v))=0$ that is $\xi(v)(u-v) \in \mc N_\mu$  for all $v \in B(0,\epsilon)$.
 We have
 \begin{align*}
     \frac{\partial }{\partial t}G_u(\xi(v),v) &= \left(\frac{n\theta}{s}+q-1 \right)(\xi(v))^{\frac{n\theta}{s}+q-2}\|u-v\|^{\frac{n\theta}{s}}\\
     &\quad\quad-(q+r-1)(\xi(v))^{q+r-2}\int_\Om f(\xi(v)(u-v)^+)(u-v)^+~dx\\
     & \quad\quad -(\xi(v))^{q+r-1}\int_\Om f^\prime(\xi(v)(u-v)^+)((u-v)^+)^2~dx
 \end{align*}
 and $\frac{\partial }{\partial t}G_u(\xi(0),0)>0$ so by continuity of the map $\frac{\partial }{\partial t}G_u$, it is possible to choose $\epsilon>0$ smaller enough such that $\frac{\partial }{\partial t}G_u(\xi(v),v)>0$ for $v \in B(0,\epsilon)$. This implies $\xi(v)(u-v)\in \mc N_\mu^+$ for $v \in B(0,\epsilon)$.
 \end{proof}

 \begin{remark}\label{KSE-remark1}
 It is easy to see that Lemma \ref{KSE-lemma4} holds if  $\mc N_\mu^+$ is replaced by $\mc N_\mu$.
 \end{remark}

 Under the assumption $\mu\in (0,\mu_0)$ and with the help of Lemma \ref{KSE-lemma2} we know that $\mc N_\mu^+ \cup \mc N_\mu^0= \mc N_\mu^+ \cup \{0\}$ is a closed set in $X_0$. By virtue of Lemma \ref{KSE-lemma3}, let us set
 \[\Upsilon_\mu^+ =\inf_{u\in \mc N_\mu^+ \cup \{0\} } I_\mu(u).\]
 Then we can apply the Ekeland variational principle to obtain a miminizing sequence $\{u_k\}\subset \mc N_\mu^+ \cup \{0\}$ which satisfies the following conditions
 \begin{align}
     \Upsilon_\mu^+ &\leq I_\mu(u_k)\leq \Upsilon^+_\mu +\frac{1}{k}\label{KSE-EVP1}\\
     I_\mu(u) &\geq I_\mu(u_k)-\frac{1}{k}\|u-u_k\|,\; \text{for all}\; u\in \mc N_\mu^+ \cup \{0\}.\label{KSE-EVP2}
 \end{align}
 Our aim is to prove that $\Upsilon_\mu^+$ is achieved by $I_\mu$ over $\mc N_\mu^+ \cup \{0\}$ which, in turn, will give us our first solution to the problem $(P_\mu)$.

\begin{lemma}\label{KSE-lemma5}
There exists a constant $C_0>0$ such that $\inf\limits_{u\in\mc N_\mu^+}I_\mu(u)\leq -C_0$. As a consequence, $\Upsilon_\mu^+ < -C_0$.
\end{lemma}
\begin{proof}
Let $u \in \mc N_{\mu}^+$ then using $\phi_u^\prime(1)=0$ we can write
\[I_\mu(u) = \left(\frac{s(1-q)-n\theta}{n\theta(1-q)} \right)\|u\|^{\frac{n\theta}{s}}-\frac{1}{1-q}\int_\Om f(u^+)u^+~dx-\int_\Om F(u^+)~dx.\]
Since $\phi_u^{\prime \prime}(1)>0$ we get
\[\left(\frac{s(1-q)-n\theta}{s} \right)\|u\|^{\frac{n\theta}{s}}< -q\int_\Om f(u^+)u^+~dx -\int_{\Om}f^\prime(u^+)(u^+)^2~dx.\]
Using this in the definition of $I_\mu$ we obtain
\begin{equation}\label{KSE-lemma5-eq1}
    I_\mu(u)< \left(\frac{\frac{n\theta}{s}-q}{\frac{n\theta}{s}(1-q)}\right)\int_\Om f(u^+)u^+~dx- \frac{s}{n\theta(1-q)} \int_\Om f^\prime (u^+)(u^+)^2~dx -\int_\Om F(u^+)~dx.
\end{equation}
Consider the function
\[\rho(l)=  \left(\frac{\frac{n\theta}{s}-q}{\frac{n\theta}{s}(1-q)}\right) f(l)l- \frac{s}{n\theta(1-q)}  f^\prime (l)l^2 - F(l), \; \text{for}\; l>0.\]
With some computations, we get that
\begin{align*}
    f^\prime(l)l &= f(l)\left( r-1 +\beta l^{\beta}\right)\\
    f^{\prime \prime}(l)l^2 &= f(l)\left[(r-1)(r-2) + \beta^2l^{2\beta}+ \beta \left(r+\beta-1\right)l^{\beta}\right].
\end{align*}
Using this, we have
\begin{align*}
    \rho^\prime(l) &= \left(\frac{\frac{n\theta}{s}-q-2}{\frac{n\theta}{s}(1-q)}\right) f^\prime(l)l- \frac{s}{n\theta(1-q)}  f^{\prime\prime} (l)l^2
    + \frac{q\left(\frac{n\theta}{s}-1\right)}{\frac{n\theta(1-q)}{s}} f(l)\\
    &= \frac{\left(\frac{n\theta}{s}-r\right)(r-1+q)}{\frac{n\theta}{s}(1-q)} f(l) + \frac{\beta\left( \frac{n\theta}{s}-q-3-r-\beta\right)}{\frac{n\theta}{s}(1-q)}f(l)l^{\beta} - \beta^2 \frac{f(l)l^{2\beta}}{\frac{n\theta}{s}(1-q)}.
\end{align*}
Since $\frac{n\theta}{s}<r$ and $ \frac{n\theta}{s}-q-3-r-\beta<0$, the first and second term in of $\rho^\prime(l)$ is negative which makes $\rho^\prime(l)\leq 0$ for all $l \in \mb R^+$. Moreover $\rho(0)=0$ asserts that $\rho(l)\leq 0$ for all $l \in \mb R^+$. Next we can verify that
\begin{align*}
    \lim_{l \to 0^+} \frac{\rho(l)}{l^r} &= -\frac{\left(r-\frac{n\theta}{s}\right)\left(r-1+q\right)}{ \frac{n\theta}{s}(1-q)}\\
    \text{and}\; \lim_{l \to +\infty}\frac{\rho(l)}{l^{r+ \beta}\exp(l^{\beta})} &= -\frac{\beta}{\frac{n\theta}{s}(1-q)}.
\end{align*}
These two estimates suggests that
\begin{align*}
    \rho(l) &\leq -\frac{\left(r-\frac{n\theta}{s}\right)\left(r-1+q\right)}{ \frac{n\theta}{s}(1-q)}l^{r}\exp( l^{\beta})- \frac{\beta}{\frac{n\theta}{s}(1-q)}l^{r+\beta}\exp( l^{\beta})\\
    &={-1} \left(\left(r-\frac{n\theta}{s}\right)\left(r-1+q\right)+ \beta l^{\beta}\right)\frac{l^r\exp( l^{\beta})}{\frac{n\theta}{s}(1-q)}.
\end{align*}
Therefore, from \eqref{KSE-lemma5-eq1} it follows that
\begin{align*}
    I_\mu(u) &< -\int_\Om\left[\left(r-\frac{n\theta}{s}\right)\left(r-1+q\right)+\beta(u^+)^{\beta}\right]\frac{(u^+)^r\exp( (u^+)^{\beta})}{\frac{n\theta}{s}(1-q)}~dx\\
    &\leq - C\int_\Om (u^+)^{r+\beta}~dx
\end{align*}
for some positive constant $C$ depending on $n,\;\theta, r,q$. Now choosing $C_0:= - C\int_\Om (u^+)^{r+\beta}~dx$, we obtain $\inf\limits_{u\in\mc N_\mu^+}I_\mu(u)\leq -C_0$.
\end{proof}

Now in subsequent results, we establish some properties of the sequence $\{u_k\}\subset \mc N_\mu^+$ obtained in \eqref{KSE-EVP1} and \eqref{KSE-EVP2} which will help us to obtain compactness of $\{u_k\}$.

\begin{lemma}\label{KSE-lemma6}
The sequence $\{u_k\}$ obtained in \eqref{KSE-EVP1} and \eqref{KSE-EVP2} is bounded in $X_0$ and there exists a $u_0\in X_0\setminus \{0\}$ such that $u_k \rightharpoonup u_0$ weakly in $X_0$.
\end{lemma}
\begin{proof}
From \eqref{KSE-EVP1}, it follows that $I_{\mu}(u_k) \to \Upsilon_\mu^+$ as $k \to \infty$. By Lemma \eqref{KSE-lemma3}, we know that $I_\mu$ is coercive so this implies that $\{u_k\}$ must be essentially bounded over $X_0$. Now from relexivity of $X_0$, we get that up to a subsequence (still denoted by $\{u_k\}$), there exists a $u_0\in X_0$ such that $u_k \rightharpoonup u_0$ weakly in $X_0$.
Using Lemma \ref{KSE-lemma5} and weak lower semi-continuity of norm, we obtain
\[I_\mu(u_0) \leq \liminf_{k\to \infty} I_\mu(u_k) = \Upsilon_\mu^+ <0\]
which implies that $u_0 \not \equiv 0$.
\end{proof}

\begin{lemma}\label{KSE-lemma7}
Let $\{u_k\}\subset \mc N_\mu^+$ be the sequence obtained in \eqref{KSE-EVP1} and \eqref{KSE-EVP2} then the following holds-
\begin{enumerate}
    \item[(i)] $\liminf\limits_{k \to \infty}\|u_k\|>0$,
    \item[(ii)]
   $ \liminf\limits_{k \to \infty}\left[\displaystyle \left(\frac{n\theta}{s}+q-1\right)\|u_k\|^{\frac{n\theta}{s}}-q\int_\Om f(u_k^+)u_k^+~dx-\int_\Om f^\prime(u_k^+)(u_k^+)^2~dx \right]>0$.
\end{enumerate}
\end{lemma}
\begin{proof}
\begin{enumerate}
    \item[(i)] From \eqref{KSE-EVP1} and Lemma \ref{KSE-lemma5}, we get that $I_\mu(u_k) \leq -C_0+ \frac{1}{k}$ for each $k \in \mb N$ which implies that $\phi_{u_k}(1)-\frac{1}{r}\phi^{ \prime}_{u_k}(1) \leq -C_0+ \frac{1}{k} $. Therefore
    \begin{align*}
        -C_0+ \frac{1}{k}& \geq \left( \frac{s}{n\theta}-\frac{1}{r}\right)\|u_k\|^{\frac{n\theta}{s}}-\mu \left(\frac{1}{1-q}-\frac{1}{r} \right)\int_\Om (u^+)^{1-q}~dx\\
        &\quad \quad+ \left(\frac{1}{r}\int_\Om f(u_k^+)u_k^+~dx-\int_\Om F(u^+_k)~dx\right)\\
        & \geq -\mu \left(\frac{r+q-1}{r(1-q)} \right)\int_\Om (u_k^+)^{1-q}~dx
    \end{align*}
    using $\frac{n\theta}{s}<r$ and $rF(l) \leq f(l)l$. Now applying H\"{o}lder inequality along with Sobolev embedding on right hand side of the above estimate, we obtain
    \[\|u_k\|^{1-q}\geq \left(C_1-\frac{1}{k}\right)\frac{r(1-q)}{\mu(r+q-1)},\]
    where $C_1$ is a positive constant. Hence we can choose $k\in \mb N$ large enough so that $\left(C_1-\frac{1}{k}\right)>0$ which will give $\liminf\limits_{k \to \infty}\|u_k\|>0$.

    \item[(ii)] We need to prove that
    \[\liminf_{k \to \infty}\left[\left(\frac{n\theta}{s}+q-1\right)\|u_k\|^{\frac{n\theta}{s}}-q\int_\Om f(u_k^+)u_k^+~dx-\int_\Om f^\prime(u_k^+)(u_k^+)^2~dx \right]>0.\]
    We assume by contradiction that
    \begin{equation}\label{KSE-lemma7-eq1}
        \left(\frac{n\theta}{s}+q-1\right)\|u_k\|^{\frac{n\theta}{s}}-q\int_\Om f(u_k^+)u_k^+~dx-\int_\Om f^\prime(u_k^+)(u_k^+)^2~dx =o_k(1)
    \end{equation}
    where $o_k(1)\to 0$ as $k \to \infty$.{
    Now using the fact that $u_k \in \mc N_\mu$ and substituting the value of $\|u_k\|^{\frac{n\theta}{s}}$ from $\phi_{u_k}^\prime(1)=0$, we have
    \begin{align}
        o_k(1)&= -\left[\left(1-\frac{n\theta}{s}\right)\int_{\Om}f(u_k^+)u_k^+~dx+ \int_\Om f^\prime(u_k^+)(u_k^+)^2~dx\right.\nonumber\\ 
        & \quad \quad\left.-\mu \left(\frac{n\theta}{s}+q-1 \right)\int_\Om(u_k^+)^{1-q}~dx\right]\label{KSE-lemma7-eq2}
    \end{align}
    It is easy to verify that
    \begin{align*}
  f^\prime(u_k^+)(u_k^+)^2= \left(r-1+\beta (u_k^+)^\beta\right)f(u_k^+)u_k^+ 
    =\left(r-1+\beta (u_k^+)^\beta\right)(u_k^+)^{r}\exp((u_k^+)^\beta).
    \end{align*}
    From the weak convergence of $u_k \rightharpoonup u_0$ as $k\to \infty$ in $X_0$ and $\beta < \frac{n}{n-s}$ we get that 
    \begin{align*}
       & \lim_{k\to \infty}\int_\Om(u_k^+)^{1-q}~dx  = \int_\Om(u_0^+)^{1-q}~dx,\\
       & \lim_{k\to \infty}\int_\Om f^\prime(u_k^+)(u_k^+)^2~dx = \int_\Om f^\prime(u_0^+)(u_0^+)^2~dx,\\
        &\text{and}\;\;\lim_{k\to \infty}\int_{\Om}f(u_k^+)u_k^+~dx= \int_{\Om}f(u_0^+)u_0^+~dx.
    \end{align*}
     Using this and  weak lower semicontinuity of norms in  \eqref{KSE-lemma7-eq1} gives us 
    \[ \left(\frac{n\theta}{s}+q-1\right)\|u_0\|^{\frac{n\theta}{s}}\leq q\int_\Om f(u_0^+)u_0^+~dx+\int_\Om f^\prime(u_0^+)(u_0^+)^2~dx\]
    implying that $u_0 \in \Gamma_0\setminus \{0\}$.
    Therefore, passing on the limit as $k \to \infty$ in \eqref{KSE-lemma7-eq2}, we get 
    \begin{align*}
        0 &= \left[\left(1-\frac{n\theta}{s}\right)\int_{\Om}f(u_0^+)u_0^+~dx+ \int_\Om f^\prime(u_0^+)(u_0^+)^2~dx\right.\nonumber\\ 
        & \quad \quad\left.-\mu \left(\frac{n\theta}{s}+q-1 \right)\int_\Om(u_0^+)^{1-q}~dx\right]\geq \Gamma_0
    \end{align*}
     which contradicts Lemma \ref{KSE-lemma1}.}
\end{enumerate}
\end{proof}

Recalling Lemma \ref{KSE-lemma4}, we infer that there exists a sequence of differentiable functions $\xi_k:B(0,\epsilon_k) \to \mb R^+$ for some $\epsilon_k>0$ satisfying
 \[\xi_k(0)=1\;\;\; \text{and}\;\;\; \xi_k(v)(u_k-v)\in \mc N_\mu^+,\; \text{for all}\; v \in B(0,\epsilon_k).\]

 \begin{lemma}\label{KSE-lemma8}
 For $\mu \in (0,\mu_0)$ and $k\in \mb N$ large enough, $|\langle \xi_k^\prime(0),\varphi\rangle|$ is uniformly bounded for any $0\leq\varphi\in X_0$, where $|\langle \xi_k^\prime(0),\varphi\rangle|$ denotes the dual action of the derivative of $\xi_k$ at zero  on $\varphi \in X_0$.
 \end{lemma}
\begin{proof}
We have the following two equations satisfied by $\xi_k(v)$, where $v\in B(0,\epsilon)$ and $u_k$
\begin{align}
    &\|u_k\|^{\frac{n\theta}{s}}-\mu \int_{\Om}(u_k^+)^{1-q}~dx-\int_{\Om}f(u_k^+)u_k^+~dx=0\label{KSE-lemma8-eq1}\\
    &(\xi_k(v))^{\frac{n\theta}{s}}\|u_k-v\|^{\frac{n\theta}{s}}-\mu (\xi_k(v))^{1-q}\int_\Om ((u_k-v)^+)^{1-q}~dx\label{KSE-lemma8-eq2}\\
    &\quad \quad-\xi_k(v) \int_\Om f(\xi_k(v)(u_k-v)^+)(u_k-v)^+~dx=0.\nonumber
\end{align}
For any $0\leq\varphi\in X_0$, we choose $t>0$ small enough so that $t\varphi \in B(0,\epsilon_k)$. Now putting $v = -t\varphi$ in \eqref{KSE-lemma8-eq2} and subtracting it from \eqref{KSE-lemma8-eq1}, we obtain
\begin{align}
    &\left(\xi_k^{\frac{n\theta}{s}}(t\varphi)-1\right)\|u_k+t\varphi\|^{\frac{n\theta}{s}}-\mu \left(\xi_k^{1-q}(t\varphi)-1\right)\int_\Om ((u_k+t\varphi)^+)^{1-q}~dx\label{KSE-lemma8-eq3}\\
    &\quad\quad + (\|u_k+t\varphi\|^{\frac{n\theta}{s}}-\|u_k\|^{\frac{n\theta}{s}})-\mu \left(\int_\Om ((u_k+t\varphi)^+)^{1-q}~dx-\int_\Om (u_k^+)^{1-q}~dx\right)\nonumber\\
   & \quad \quad \quad \quad -\int_\Om \left(f(\xi_k(t\varphi)(u_k+t\varphi)^+)(u_k+t\varphi)^+-f(u_k^+)(u_k^+)\right)~dx=0.\nonumber
\end{align}
We know the following-
\begin{enumerate}
    \item[(i)] $\displaystyle \lim_{t\to0^+}\left(\frac{\xi_k^{\frac{n\theta}{s}}(t\varphi)-1}{t}\right)=\frac{n\theta}{s}\xi_k^{\frac{n\theta}{s}-1}(0)\langle \xi_k^\prime(0),\varphi\rangle =\frac{n\theta}{s}\langle \xi_k^\prime(0),\varphi\rangle$,
    \item[(ii)] $\displaystyle\lim_{t\to0^+} \frac{\|u_k+t\varphi\|^{\frac{n\theta}{s}}-\|u_k\|^{\frac{n\theta}{s}}}{t}= \frac{n\theta}{s}\|u_k\|^{\frac{n}{s}(\theta-1)}\mc A(u_k,\varphi)$, where
    \[\mc A(w_1,w_2):= \int_{\mb R^{2n}}\frac{|w_1(x)-w_1(y)|^{\frac{n}{s}-2}(w_1(x)-w_1(y))(w_2(x)-w_2(y))}{|x-y|^{2n}}~dxdy.,\]
    \item[(iii)] We estimate the exponential term as follows-
    \begin{align*}
       & \lim_{t\to0^+}\int_\Om \frac{f(\xi_k(t\varphi)(u_k+t\varphi)^+)(u_k+t\varphi)^+-f(u_k^+)(u_k^+)}{t}\\
       & = \lim_{t\to0^+}\int_\Om  \frac{f(\xi_k(t\varphi)(u_k+t\varphi)^+)(u_k+t\varphi)^+-f((u_k+t\varphi)^+)(u_k+t\varphi)^+}{t}\\
       & \quad \quad+ \lim_{t\to0^+}\int_\Om  \frac{f((u_k+t\varphi)^+)(u_k+t\varphi)^+ -f(u_k^+)(u_k^+)}{t}\\
       &=  \langle \xi_k^\prime(0),\varphi\rangle\int_\Om  \left(f^\prime(u_k^+)(u_k^+)^2+f(u_k^+)u_k^+\right)~dx + \int_\Om (f^\prime(u_k^+)u_k^+ + f(u_k^+))\varphi~dx.
    \end{align*}
\end{enumerate}
Using the fact that ${\displaystyle\int_\Om ((u_k+t\varphi)^+)^{1-q}~dx-\int_\Om (u_k^+)^{1-q}~dx }\geq 0 $ in \eqref{KSE-lemma8-eq3}, dividing it by $t>0$, passing through the limit as $t\to 0^+$ and using $(i)$, $(ii)$ and $(iii)$ above, we get
\begin{align*}
    0 &\leq \frac{n\theta}{s}\left(\langle\xi_k^\prime(0),\varphi \rangle\|u_k\|^{\frac{n\theta}{s}}+ \|u_k\|^{\frac{n}{s}(\theta-1)}\mc A(u_k,\varphi)\right) -\mu(1-q)\langle\xi_k^\prime(0),\varphi \rangle\int_\Om (u_k^+)^{1-q}~dx\\
    &\quad \quad-\langle \xi_k^\prime(0),\varphi\rangle\int_\Om  \left(f^\prime(u_k^+)(u_k^+)^2+f(u_k^+)u_k^+\right)~dx
    -\int_\Om (f^\prime(u_k^+)u_k^++ f(u_k^+))\varphi~dx\\
    &= \langle\xi_k^\prime(0),\varphi \rangle\left(  \left(\frac{n\theta}{s}+q-1\right)\|u_k\|^{\frac{n\theta}{s}}+\int_\Om \left(-qf(u_k^+)u_k^+ - f^\prime(u_k^+)(u_k^+)^2\right)~dx \right)\\
    & \quad \quad+ \frac{n\theta}{s}\|u_k\|^{\frac{n}{s}(\theta-1)}\mc A(u_k,\varphi) -\int_\Om (f^\prime(u_k^+)u_k^++ f(u_k^+))\varphi~dx\\
    &\leq \langle\xi_k^\prime(0),\varphi \rangle\left(  \left(\frac{n\theta}{s}+q-1\right)\|u_k\|^{\frac{n\theta}{s}}+\int_\Om \left(-qf(u_k^+)u_k^+ - f^\prime(u_k^+)(u_k^+)^2\right)~dx \right)\\
    & \quad \quad+ \frac{n\theta}{s}\|u_k\|^{\frac{n}{s}(\theta-1)}\mc A(u_k,\varphi).
\end{align*}
where the second line is obtained using the fact that $u_k\in \mc N_\mu$.
From this estimate, it follows that
\[ \langle\xi_k^\prime(0),\varphi \rangle \geq  \frac{-\displaystyle\frac{n\theta}{s}\|u_k\|^{\frac{n}{s}(\theta-1)}\mc A(u_k,\varphi) }{  \left(\frac{n\theta}{s}+q-1\right)\|u_k\|^{\frac{n\theta}{s}}+\displaystyle\int_\Om \left(-qf(u_k^+)u_k^+ - f^\prime(u_k^+)(u_k^+)^2\right)~dx }.\]
Therefore, using $\mc A(u_k,\varphi)\leq \|u_k\|^{\frac{n}{s}-1}\|\varphi\|$, boundedness of $\{u_k\}$ in $X_0$ and Lemma \ref{KSE-lemma7}$(ii)$,
 we deduce that $\liminf\limits_{k \to \infty}\;\langle\xi_k^\prime(0),\varphi \rangle \neq -\infty$. We  now aim to establish that even $\liminf\limits_{k \to \infty}\;\langle\xi_k^\prime(0),\varphi \rangle \neq +\infty$ which will essentially complete the proof. It is easy to verify that
 \[|\xi_k(t\varphi)-1|\|u_k\|+\xi_k(t\varphi)\|t\varphi\| \geq \|\xi_k(t\varphi)(u_k+t\varphi)-u_k\|\]
 and for large enough $k$, we may assume that $\xi_k(t\varphi)\geq \xi_k(0)=1$. Therefore using $u_k,\; \xi_k(t\varphi)(u_k+t\varphi)\in \mc N_\mu^+$, we obtain the following estimate
 \begin{equation}\label{KSE-lemma8-eq4}
 \begin{split}
     &(\xi_k(t\varphi)-1)\frac{\|u_k\|}{k}+\xi_k(t\varphi) \frac{\|t\varphi\|}{k} \\
     &\geq \frac{\|\xi_k(t\varphi)(u_k+t\varphi)-u_k\|}{k}
      \geq I_\mu(u_k)- I_\mu(\xi_k(t\varphi)(u_k+t\varphi))\\
     & =\left(\frac{s}{n\theta}-\frac{\mu}{1-q} \right)\left(\|u_k\|^{\frac{n\theta}{s}}-\|u_k+t\varphi\|^{\frac{n\theta}{s}}\right)+ \left(\frac{s}{n\theta}-\frac{\mu}{1-q} \right) \|u_k+t\varphi\|^{\frac{n\theta}{s}}\left(1-\xi_k^{\frac{n\theta}{s}}(t\varphi)\right)\\
     &\quad \quad + \frac{\mu}{1-q}\int_\Om \left(f(u_k^+)u_k^+- f(\xi_k(t\varphi)(u_k+t\varphi)^+)\xi_k(t\varphi)(u_k+t\varphi)^+\right)~dx\\
     & \quad \quad \quad + \int_\Om\left(F(\xi_k(t\varphi)(u_k+t\varphi))-F(u_k) \right)~dx.
     \end{split}
 \end{equation}
 The following is easy to verify with some computations-
 \begin{enumerate}
     \item[(i)] $\lim\limits_{t\to 0^+}\displaystyle\int_\Om\frac{f(u_k^+)u_k^+- f(\xi_k(t\varphi)(u_k+t\varphi)^+)\xi_k(t\varphi)(u_k+t\varphi)^+}{t}\\
     = -\langle \xi_k^\prime(0),\varphi\rangle \int_\Om \left(f^\prime(u_k^+)(u_k^+)^2+f(u_k^+)u_k^+\right)~dx - \int_\Om \left(f^\prime(u_k^+)u_k^++f(u_k^+)\right)\varphi~dx$,
     \item[(ii)]$\lim\limits_{t\to 0^+}\displaystyle\int_\Om\frac{F(\xi_k(t\varphi)(u_k+t\varphi))-F(u_k)}{t}= \langle \xi_k^\prime(0),\varphi\rangle\int_\Om f(u_k^+)u_k^+~dx+ \int_\Om f(u_k^+)\varphi~dx$.
 \end{enumerate}
 Therefore dividing \eqref{KSE-lemma8-eq4} by $t>0$ and passing through limit $t\to 0^+$ we obtain
 \begin{align*}
     &(\xi_k(t\varphi)-1)\frac{\|u_k\|}{k}+\xi_k(t\varphi) \frac{\|t\varphi\|}{k} \\
     & \leq \langle \xi_k^\prime(0),\varphi\rangle \left[ - \left(\frac{s}{n\theta}- \frac{1}{1-q}\right)\frac{n\theta}{s}\|u_k\|^{\frac{n\theta}{s}}-\frac{1}{1-q}\int_\Om \left(f^\prime(u_k^+)(u_k^+)^2+qf(u_k^+)u_k^+\right)~dx\right]\\
     & \quad \quad + \left[ - \left(\frac{s}{n\theta}- \frac{1}{1-q}\right)\frac{n\theta}{s}\|u_k\|^{\frac{n}{s}(\theta-1)}\mc A(u_k,\varphi)-\frac{1}{1-q}\int_\Om \left(f^\prime(u_k^+)u_k^+ +qf(u_k^+)\right)\varphi~dx\right].
 \end{align*}
 This implies that
     \begin{align*}
         \frac{\|\varphi\|}{k} &\geq \frac{\langle \xi_k^\prime(0),\varphi\rangle}{1-q} \left[ - \left(\frac{n\theta}{s}- 1+q\right)\|u_k\|^{\frac{n\theta}{s}}-\int_\Om \left(f^\prime(u_k^+)(u_k^+)^2+qf(u_k^+)u_k^+\right)~dx-\frac{(1-q)\|u_k\|^{\frac{n\theta}{s}}}{k}\right]\\
         & \quad \quad +  \frac{1}{1-q} \left(\frac{n\theta}{s}- 1+q\right)\|u_k\|^{\frac{n}{s}(\theta-1)}\mc A(u_k,\varphi)-\frac{1}{1-q}\int_\Om \left(f^\prime(u_k^+)u_k^+ +qf(u_k^+)\right)\varphi~dx.
     \end{align*}
     From this estimate, we can conclude that $\liminf\limits_{k \to \infty}\;\langle\xi_k^\prime(0),\varphi \rangle \neq +\infty$ because otherwise boundedness of $\{u_k\}$ in $X_0$ and Lemma \ref{KSE-lemma7}$(ii)$ will give a contradiction. This completes the proof.
\end{proof}

\begin{lemma}\label{KSE-lemma9}
Let $\mu \in (0,\mu_0)$ and $\{u_k\}\subset \mc N_\mu^+$ be the sequence obtained in \eqref{KSE-EVP1} and \eqref{KSE-EVP2} then for any $\varphi \in X_0$, the following holds
\begin{enumerate}
    \item[(i)] $(u_k^+)^{-q}\varphi \in L^1(\Om)$,
    \item[(ii)] As $k \to \infty$, we have
   \begin{align*}
   & \left(\|u_k\|^{\frac{n}{s}(\theta-1)}\right)\int_{\mb R^{2n}}\frac{|u_k(x)-u_k(y)|^{\frac{n}{s}}(u_k(x)-u_k(y)) (\varphi(x)-\varphi(y))}{|x-y|^{2n}}~dxdy\\
    & \quad \quad - \mu \int_\Om (u_k^+)^{-q}\varphi~dx - \int_{\Om}f(u_k^+)\varphi~dx =o_k(1).
    \end{align*}
\end{enumerate}
\end{lemma}
\begin{proof}
At first we assume $0\leq \varphi\in X_0$ and recalling \eqref{KSE-lemma8-eq4}, we consider
\begin{align*}
     &(\xi_k(t\varphi)-1)\frac{\|u_k\|}{k}+\xi_k(t\varphi) \frac{\|t\varphi\|}{k} \geq I_\mu(u_k)- I_\mu(\xi_k(t\varphi)(u_k+t\varphi))\\
     & =-\frac{s}{n\theta}\left(\xi_k^{\frac{n\theta}{s}}(t\varphi)-1\right)\|u_k\|^{\frac{n\theta}{s}}+ \frac{s}{n\theta}\xi_k^{\frac{n\theta}{s}}(t\varphi)\left(\|u_k\|^{\frac{n\theta}{s}}-\|u_k+t\varphi\|^{\frac{n\theta}{s}}\right)\\
      &\quad \quad+\frac{\mu}{1-q}\left(\xi_k^{1-q}(t\varphi)-1\right)
    \int_\Om((u_k+t\varphi)^+)^{1-q}~dx
      +\frac{\mu}{1-q}\int_\Om\left(((u_k+t\varphi)^+)^{1-q}-(u_k^+)^{1-q}\right)~dx\\
      &\quad \quad \quad+ \int_\Om \left(F(\xi_k(t\varphi) (u_k+t\varphi))-F(u_k)\right)~dx.
\end{align*}
On dividing the above estimate by $t>0$ and then passing through the limit as $t\to 0^+$, we obtain
\begin{align*}
    &\langle \xi_k^\prime(0),\varphi\rangle \frac{\|u_k\|}{k}+ \frac{\|\varphi\|}{k}\\
    &\geq - \langle \xi_k^\prime(0),\varphi\rangle\left[ \|u_k\|^{\frac{n\theta}{s}}-\mu \int_\Om(u_k^+)^{1-q}~dx-f(u_k^+)u_k^+~dx\right]- \|u_k\|^{\frac{n}{s}(\theta-1)}\mc A(u_k,\varphi)\\
    &\quad\quad + \int_\Om f(u_k^+)\varphi~dx+\frac{\mu}{1-q}\liminf_{t\to 0^+}\int_\Om\frac{((u_k+t\varphi)^+)^{1-q}-(u_k^+)^{1-q}}{t}~dx\\
    & = - \|u_k\|^{\frac{n}{s}(\theta-1)}\mc A(u_k,\varphi)+ \int_\Om f(u_k^+)\varphi~dx +\frac{\mu}{1-q}\liminf_{t\to 0^+}\int_\Om\frac{((u_k+t\varphi)^+)^{1-q}-(u_k^+)^{1-q}}{t}~dx,
\end{align*}
since $u_k \in \mc N_\mu$. Since $((u_k+t\varphi)^+)^{1-q}-(u_k^+)^{1-q}\geq 0$, using Fatou's Lemma we infer that
\[\liminf_{t\to 0^+}\int_\Om\frac{((u_k+t\varphi)^+)^{1-q}-(u_k^+)^{1-q}}{t}~dx \; {\geq} \;(1-q)\int_\Om (u_k^+)^{-q}\varphi~dx.\]
Therefore using the above estimate, we conclude that
\begin{align*}
    \mu \int_\Om (u_k^+)^{-q}\varphi~dx &\leq  \frac{\langle \xi_k^\prime(0),\varphi\rangle\|u_k\|+\|\varphi\|}{k}+  \|u_k\|^{\frac{n}{s}(\theta-1)}\mc A(u_k,\varphi)-\int_\Om f(u_k^+)\varphi~dx \\
    & \leq \frac{\langle \xi_k^\prime(0),\varphi\rangle\|u_k\|+\|\varphi\|}{k}+  \|u_k\|^{\frac{n}{s}(\theta-1)}\mc A(u_k,\varphi),
\end{align*}
from where $(i)$ follows using the boundedness of $\{u_k\}$ in $X_0$. Since it holds
\begin{align*}
    \langle \xi_k^\prime(0),\varphi\rangle \frac{\|u_k\|}{k}+ \frac{\|t\varphi\|}{k}\geq - \|u_k\|^{\frac{n}{s}(\theta-1)}\mc A(u_k,\varphi)+ \int_\Om f(u_k^+)\varphi~dx +\mu \int_\Om (u_k^+)^{-q}\varphi~dx,
\end{align*}
using Lemma \ref{KSE-lemma8} and passing $k \to \infty$ in the above estimate, we get
\begin{equation}\label{KSE-lemma9-eq1}
    \|u_k\|^{\frac{n}{s}(\theta-1)}\mc A(u_k,\varphi)- \int_\Om f(u_k^+)\varphi~dx -\mu \int_\Om (u_k^+)^{-q}\varphi~dx \geq o_k(1).
\end{equation}
Now we aim at proving that \eqref{KSE-lemma9-eq1} holds for any $\varphi\in X_0$ which will essentially establish $(ii)$. So let $\psi \in X_0$ and $\epsilon>0$ be given then we define $w_\epsilon := u_k^++\epsilon \psi$. We know that \eqref{KSE-lemma9-eq1} holds true by taking $\varphi = w_\epsilon^+$ which gives us
\begin{equation}\label{KSE-lemma9-eq2}
    o_k(1)\leq \|u_k\|^{\frac{n}{s}(\theta-1)}\mc A(u_k,w_\epsilon+w_\epsilon^-)- \int_\Om f(u_k^+)(w_\epsilon+w_\epsilon^-)~dx -\mu \int_\Om (u_k^+)^{-q}(w_\epsilon+w_\epsilon^-)~dx.
\end{equation}
We can rewrite the following term as
\[\mc A(u_k,w_\epsilon+w_\epsilon^-)= \mc A(u_k,u_k^+)+ \epsilon\mc A(u_k,\psi)+ A(u_k,w_\epsilon^-) \leq \|u_k\|^{\frac{n}{s}}+ \epsilon\mc A(u_k,\psi)+ A(u_k,w_\epsilon^-), \]
using the fact that
\[(u_k(x)-u_k(y))(u_k^+(x)-u_k^+(y))\leq |u_k(x)-u_k(y)|^2,\;\text{for a.e.}\; x,y\in \mb R^n.\]
With this, \eqref{KSE-lemma9-eq2} reduces to
\begin{equation}\label{KSE-lemma9-eq3}
\begin{split}
     o_k(1) &\leq \left[ \|u_k\|^{\frac{n\theta}{s}}-\mu \int_\Om(u_k^+)^{1-q}~dx-f(u_k^+)u_k^+~dx\right]\\
     &\quad\quad+ \epsilon \left[\|u_k\|^{\frac{n}{s}(\theta-1)}\mc A(u_k,\psi)- \int_\Om f(u_k^+)\psi~dx -\mu \int_\Om (u_k^+)^{-q}\psi~dx\right]\\
     & \quad \quad + \left[\|u_k\|^{\frac{n}{s}(\theta-1)}\mc A(u_k,w_\epsilon^-)- \int_\Om f(u_k^+)w_\epsilon^-~dx -\mu \int_\Om (u_k^+)^{-q}w_\epsilon^-~dx\right]\\
     &=\epsilon \left[\|u_k\|^{\frac{n}{s}(\theta-1)}\mc A(u_k,\psi)- \int_\Om f(u_k^+)\psi~dx -\mu \int_\Om (u_k^+)^{-q}\psi~dx\right]\\
     & \quad \quad + \left[\|u_k\|^{\frac{n}{s}(\theta-1)}\mc A(u_k,w_\epsilon^-)- \int_\Om f(u_k^+)w_\epsilon^-~dx -\mu \int_\Om (u_k^+)^{-q}w_\epsilon^-~dx\right]\\
     & \leq \epsilon \left[\|u_k\|^{\frac{n}{s}(\theta-1)}\mc A(u_k,\psi)- \int_\Om f(u_k^+)\psi~dx -\mu \int_\Om (u_k^+)^{-q}\psi~dx\right] + \|u_k\|^{\frac{n}{s}(\theta-1)}\mc A(u_k,w_\epsilon^-)
     \end{split}
\end{equation}
where the second line is due to $u_k \in \mc N_\mu$ and the third inequality is due to $w_\epsilon^-\geq 0$ on $\Om$. Now we claim that
\[\lim_{\epsilon\to 0^+}\mc A(u_k,w_\epsilon^-)=0.\]
We know that $w_\epsilon^-=-(u_k^+ +\epsilon \psi)$ on $\Omega_\epsilon^- :=\{x\in \Om:\; w_\epsilon\leq 0\}$ and for a.e. $x,y\in \mb R^n$,
\[(u_k(x)-u_k(y))(u_k^+(x)-u_k^+(y))\geq |u_k^+(x)-u_k^+(y)|^2.\]
Setting
\[g_{k,\epsilon}(x,y)=\frac{ |u_k(x)-u_k(y)|^{\frac{n}{s}-2}(u_k(x)-u_k(y))(w_\epsilon^-(x)-w_\epsilon^-(y))}{|x-y|^{2n}},\]
\[\text{and}\; h_k(x,y)= \frac{ |u_k(x)-u_k(y)|^{\frac{n}{s}-2}(u_k(x)-u_k(y))(\psi(x)-\psi(y))}{|x-y|^{2n}} \]
 we note that
\begin{equation}\label{KSE-lemma9-eq4}
\begin{split}
    \mc A(u_k,w_\epsilon^-) & = \int_{\Om_\epsilon^-}\int_{\Om_\epsilon^-} g_{k,\epsilon}(x,y)~dxdy + 2 \int_{\Om_\epsilon^-}\int_{\mb R^n \setminus \Om_\epsilon^-}g_{k,\epsilon}(x,y)~dxdy \\
    & \leq -\int_{\Om_\epsilon^-}\int_{\Om_\epsilon^-}\frac{ |u_k(x)-u_k(y)|^{\frac{n}{s}-2}|u_k^+(x)-u_k^+(y)|^2}{|x-y|^{2n}}~dxdy-\epsilon \int_{\Om_\epsilon^-}\int_{\Om_\epsilon^-} h_k(x,y)dxdy\\
    &\quad -2\int_{\Om_\epsilon^-}\int_{\mb R^n \setminus\Om_\epsilon^-}\frac{ |u_k(x)-u_k(y)|^{\frac{n}{s}-2}|u_k^+(x)-u_k^+(y)|^2}{|x-y|^{2n}}~dxdy\\
    &\quad\quad\quad-2\epsilon \int_{\Om_\epsilon^-}\int_{\mb R^n \setminus\Om_\epsilon^-} h_k(x,y)dxdy\\
    & \leq -\epsilon \left(\int_{\Om_\epsilon^-}\int_{\Om_\epsilon^-} h_k(x,y)~dxdy+2 \int_{\Om_\epsilon^-}\int_{\mb R^n \setminus\Om_\epsilon^-} h_k(x,y)~dxdy \right)\\
    & \leq \epsilon \left(\int_{\Om_\epsilon^-}\int_{\Om_\epsilon^-} |h_k(x,y)|~dxdy+2 \int_{\Om_\epsilon^-}\int_{\mb R^n \setminus\Om_\epsilon^-}| h_k(x,y)|~dxdy \right).
    \end{split}
\end{equation}
Using H\"{o}lder inequality and boundedness of $\{\|u_k\|\}$, we get
\begin{equation}\label{KSE-lemma9-eq5}
\int_{\Om_\epsilon^-}\int_{\Om_\epsilon^-} |h_k(x,y)|~dxdy  \leq C  \left( \int_{\Om_\epsilon^-}\int_{\Om_\epsilon^-} \frac{|\psi(x)-\psi(y)|^{\frac{n}{s}}}{|x-y|^{2n}}~dxdy\right)^{\frac{s}{n}} \end{equation}
where $C>0$ is a constant. Since meas$(\Om_\epsilon^-)\to 0$ as $\epsilon\to 0^+$ and $\psi \in X_0\subset L^{\frac{n}{s}}(\mb R^n)$, so \eqref{KSE-lemma9-eq5} suggests that
\begin{equation}\label{KSE-lemma9-eq6}
\lim_{\epsilon \to 0^+}\int_{\Om_\epsilon^-}\int_{\Om_\epsilon^-} |h_k(x,y)|~dxdy  =0.
\end{equation}
Similarly we also have
\begin{equation}\label{KSE-lemma9-eq7}
    \int_{\Om_\epsilon^-}\int_{\mb R^n \setminus \Om_\epsilon^-} |h_k(x,y)|~dxdy  \leq C  \left( \int_{\Om_\epsilon^-}\int_{\mb R^n \setminus\Om_\epsilon^-} \frac{|\psi(x)-\psi(y)|^{\frac{n}{s}}}{|x-y|^{2n}}~dxdy\right)^{\frac{s}{n}}.
\end{equation}
Since $\psi \in X_0\subset L^{\frac{n}{s}}(\mb R^n)$,
{we have $\frac{(\psi(x)-\psi(y))}{|x-y|^{2s}}\in L^{\frac{n}{s}}(\mathbb{R}^{2n})$.}
Then for any $\delta>0$ there exists a $R_\delta>0$ large enough such that
\[ \left( \int_{Supp(\psi)}\int_{\mb R^n \setminus B(0,R_\delta)} \frac{|\psi(x)-\psi(y)|^{\frac{n}{s}}}{|x-y|^{2n}}~dxdy\right)^{\frac{s}{n}}< \frac{\delta}{2}.\]
Also, $\Om_\epsilon^-\subset Supp(\psi)$ and meas$(\Om_\epsilon^-\times B(0,R_\delta)) \to 0$ as $\epsilon\to 0^+$ which implies that there exists a $\rho_\delta>0$ and $\epsilon_\delta>0$ such that
\[meas(\Om_\epsilon^- \times B(0,R_\delta))< \rho_\delta\;\; \text{and}\;\; \left( \int_{\Om_\epsilon^-}\int_{ B(0,R_\delta)} \frac{|\psi(x)-\psi(y)|^{\frac{n}{s}}}{|x-y|^{2n}}~dxdy\right)^{\frac{s}{n}}< \frac{\delta}{2}\]
when $\epsilon \in (0,\epsilon_\delta)$. Combining these, we obtain
\[\left( \int_{\Om_\epsilon^-}\int_{ \mb R^n} \frac{|\psi(x)-\psi(y)|^{\frac{n}{s}}}{|x-y|^{2n}}~dxdy\right)^{\frac{s}{n}}< {\delta}\]
for $\epsilon \in (0,\epsilon_\delta)$.
As a consequence, putting these in \eqref{KSE-lemma9-eq7}, we have
\begin{equation}\label{KSE-lemma9-eq8}
    \lim_{\epsilon\to 0^+} \int_{\Om_\epsilon^-}\int_{\mb R^n \setminus \Om_\epsilon^-} |h_k(x,y)|~dxdy  =0.
\end{equation}
Finally dividing \eqref{KSE-lemma9-eq4}  by $\epsilon>0$, passing $\epsilon \to 0^+$ and using \eqref{KSE-lemma9-eq6} and \eqref{KSE-lemma9-eq8}, we get
\begin{equation}\label{KSE-lemma9-eq9}
    \lim_{\epsilon\to 0^+}\mc A(u_k,w_\epsilon^-)\leq 0.
\end{equation}
Hence dividing \eqref{KSE-lemma9-eq3} by $\epsilon>0$ and passing through the limit $\epsilon\to 0^+$, we get as $k \to \infty$
\[o_k(1)\leq \|u_k\|^{\frac{n}{s}(\theta-1)}\mc A(u_k,\psi)- \int_\Om f(u_k^+)\psi~dx -\mu \int_\Om (u_k^+)^{-q}\psi~dx\]
which finishes the proof of $(ii)$ due to arbitrariness of $\psi$.
\end{proof}

We now establish the existence of first solution to $(P_{\mu})$.

\begin{theorem}\label{KSE-first-sol}
$(P_\mu)$ admits a non negative weak solution $u_0 \in \mc N_\mu^+$ when $\mu \in (0,\mu_0)$ such that $I_\mu(u_0)= \Upsilon^+_\mu$.
\end{theorem}
\begin{proof}
Let $\mu \in (0,\mu_0)$ and $\{u_k\}\subset \mc N_\mu^+$ be the sequence obtained in \eqref{KSE-EVP1} and \eqref{KSE-EVP2}. Then from Lemma \ref{KSE-lemma6}, we know that the sequence $\{u_k\}$ is bounded in $X_0$ and there exists a $u_0\in X_0\setminus \{0\}$ such that $u_k\rightharpoonup u_0$ weakly in $X_0$, strongly in $L^p(\Om)$ for $p \in [1,\infty)$ and $u_k \to u_0$ point wise a.e. in $\Om$. Taking $\varphi = u_k^-$ in Lemma \ref{KSE-lemma9}$(ii)$ we get
  \begin{align*}
   & \lim_{k \to \infty}\left[\left(\|u_k\|^{\frac{n}{s}(\theta-1)}\right)\int_{\mb R^{2n}}\frac{|u_k(x)-u_k(y)|^{\frac{n}{s}}(u_k(x)-u_k(y)) (u_k^-(x)-u_k^-(y))}{|x-y|^{2n}}~dxdy\right]=0.
    \end{align*}
    Using the inequality
    \[(u_k(x)-u_k(y))(u_k^-(x)-u_k^-(y))\leq -|u_k^-(x)-u_k^-(y)|^2\]
    for a.e. $x,y\in \mb R^n$ in the above equation, we get
    \[\lim_{k \to \infty}\left[\left(\|u_k\|^{\frac{n}{s}(\theta-1)}\right)\int_{\mb R^{2n}}\frac{|u_k(x)-u_k(y)|^{\frac{n}{s}} |u_k^-(x)-u_k^-(y)|^2}{|x-y|^{2n}}~dxdy\right]\leq 0.\]
It is easy to verify that for a.e. $x,y$ in $\mb R^n$,
\[|u_k(x)-u_k(y)|^{\frac{n}{s}} |u_k^-(x)-u_k^-(y)|^2\geq |u_k^-(x)-u_k^-(y)|^{\frac{n}{s}}\]
which on inserting in the above equation gives
\[\lim_{k\to \infty}  \|u_k\|^{\frac{n}{s}(\theta-1)}\|u_k^-\|^{\frac{n}{s}}\leq 0.\]
This implies that $\lim_{k \to \infty}\|u_k^-\|=0$, since $\lim_{k \to \infty}\|u_k\|\neq 0$ in view of Lemma \ref{KSE-lemma6}. Therefore, w.l.o.g. we assume that the sequence $\{u_k\}$ is infact a sequence of non-negative functions and so we can assume that $u_0\geq 0$ in $\mb R^n$.
Using Theorem $4.9$ of \cite{Brezis-book}, we get that there exists a $h \in L^p(\Om)$ for fixed $p\in [1,\infty)$ such that $u_k \leq h$ for all $k\in \mb N$. This implies that $u_k^{1-q}\leq h^{1-q}$ for all $k\in \mb N$, therefore using Lebesgue dominated convergence theorem along with H\"{o}lder inequality it follows that
\begin{equation}\label{KSE-prop1-eq1}
    \lim_{k \to \infty} \int_\Om u_k^{1-q}~dx = \int_\Om u_0^{1-q}~dx.
\end{equation}
Moreover, since $u_k^{-q}\varphi \in L^1(\Om)$ for any $\varphi \in X_0$ using Lemma \ref{KSE-lemma9}$(i)$, we use the Fatou's Lemma to get
\begin{equation}\label{KSE-prop1-eq3}
    \int_\Om u_0^{1-q}~dx \leq \liminf_{k \to \infty}\int_\Om u_k^{-q}u_0~dx.
\end{equation}
Suppose $\|u_k\| \to d>0$ as $k \to \infty$. By a.e. pointwise convergence of $u_k$ to $u_0$ and Brezis Lieb Lemma, we get that
\begin{equation}\label{KSE-prop1-eq4}
    \|u_k\|^{\frac{n}{s}} = \|u_k-u_0\|^{\frac{n}{s}}+\|u_0\|^{\frac{n}{s}}+o_k(1).
\end{equation}
Also since $\beta<\frac{n}{n-s}$, we can use compactness of Moser-Trudinger embedding to get that
\begin{equation}\label{KSE-prop1-eq5}
  \lim_{k\to \infty}\int_\Om f(u_k)(u_k-u_0)~dx =0.
\end{equation}
Now from \eqref{KSE-prop1-eq4}, \eqref{KSE-prop1-eq5} and Lemma \ref{KSE-lemma9}, it follows that as $k \to \infty$
 \begin{align*}
   o_k(1) &= \left(\|u_k\|^{\frac{n}{s}(\theta-1)}\right)\int_{\mb R^{2n}}\frac{|u_k(x)-u_k(y)|^{\frac{n}{s}}(u_k(x)-u_k(y)) ((u_k-u_0)(x)-(u_k-u_0)(y))}{|x-y|^{2n}}~dxdy\\
    & \quad \quad - \mu \int_\Om (u_k)^{-q}(u_k-u_0)~dx - \int_{\Om}f(u_k)(u_k-u_0)dx \\
    & = \left(d^{\frac{n}{s}(\theta-1)}\right)(d^{\frac{n}{s}}- \|u_0\|^{\frac{n}{s}}) - \mu \int_\Om (u_k)^{-q}(u_k-u_0)~dx\\
    & =  \left(d^{\frac{n}{s}(\theta-1)}\right)\|u_k-u_0\|^{\frac{n}{s}} - \mu \int_\Om (u_k)^{-q}(u_k-u_0)~dx \geq \left(d^{\frac{n}{s}(\theta-1)}\right)\|u_k-u_0\|^{\frac{n}{s}},
    \end{align*}
    where we used \eqref{KSE-prop1-eq1} and \eqref{KSE-prop1-eq3} to get the last inequality. Therefore it must be $d=0$ or $\lim\limits_{k\to \infty}\|u_k-u_0\|^{\frac{n}{s}}=0$ which anyway suggests that $u_k \to u_0$ strongly in $X_0$.  From this, considering $\{u_k\}\subset \mc N_{\mu}^+ \subset \mc N_\mu$ we infer that $u_0 \in \mc N_\mu$. Passing limits in Lemma \ref{KSE-lemma7} (ii), we obtain that
    $$  \left(\frac{n\theta}{s}+q-1\right)\|u_0\|^{\frac{n\theta}{s}}-q\int_\Om f(u_0)u_0~dx-\int_\Om f^\prime(u_0)(u_0)^2~dx >0$$
    that is $u_0 \in \mc N_\mu^+$. We also deduce that
    \[\Upsilon^+_\mu \leq I_\mu(u_0) = \lim_{k\to \infty}I_{\mu}(u_k)= \upsilon_\mu^+\]
    that is $u_0$ achieves $\Upsilon_\mu^+$.

    Now we will show that $u_0$ is a positive weak solution of $(P_\mu)$. Passing on the limits in Lemma \ref{KSE-lemma9} (ii) and using Fatou's lemma as in \eqref{KSE-prop1-eq3}, we infer that for any $0\leq \varphi \in X_0$ it holds that
    \begin{align}
   & \left(\|u_0\|^{\frac{n}{s}(\theta-1)}\right)\int_{\mb R^{2n}}\frac{|u_0(x)-u_0(y)|^{\frac{n}{s}}(u_0(x)-u_0(y)) (\varphi(x)-\varphi(y))}{|x-y|^{2n}}~dxdy \nonumber\\
    & \quad \quad - \mu \int_\Om u_0^{-q}\varphi~dx - \int_{\Om}f(u_0)\varphi~dx \geq 0.\label{KSE-prop1-eq6}
    \end{align}
    So we take any $\varphi \in X_0$, then we define  $w_\epsilon:= u_0^+ +\epsilon \varphi$ for $\epsilon>0$ and consider $w_\epsilon^+$ as test function in \eqref{KSE-prop1-eq6}. Now repeating the steps \eqref{KSE-lemma9-eq2} to \eqref{KSE-lemma9-eq9} with $u_k$ replaced by $u_0$, we can establish that \eqref{KSE-prop1-eq6} holds for any $\varphi \in X_0$.  Further, \eqref{KSE-prop1-eq6} gives that $u_0^{-q}\varphi \in L^1(\Om)$ for any $\varphi \in X_0$ and thus arbitrariness of $\varphi \in X_0$ along with $u_0$ satisfying \eqref{KSE-prop1-eq6} gives us that it solves $(P_\mu)$ weakly. We already have $u_0\geq 0$ a.e. in $\mb R^n$ and since $0\notin \mc N_\mu^+$, $u_0 \not \equiv 0$. This completes the proof.
\end{proof}


\section{Existence of second solution}
This section is devoted to solving a minimization problem on $\mc N_{\mu}^-$ and hence establishing existence of second solution to $(P_\mu)$.

On a similar note to Lemma \ref{KSE-lemma4}, we also have the following Lemma.
 \begin{lemma}\label{KSE-lemma10}
 For every $u\in \mc N_\mu^-$ and $\mu \in (0,\mu_0)$, there exists a positive real number $\epsilon$ and a differentiable function $\tilde \xi : B(0,\epsilon)\subset X_0 \to \mb R^+$ such that
 \[\tilde \xi(v)>0,\; \tilde \xi(0)=1,\; \tilde \xi(v)(u-v)\in \mc N_\mu^-,\;\text{for all}\; v \in B(0,\epsilon).\]
 \end{lemma}
\begin{proof}
For $u \in \mc N_\mu^-$, we have $\phi_u^\prime(1)=0$ and $\phi_u^{\prime\prime}(1)<0$.
This gives
\[\left(\frac{n\theta}{s}+q-1\right)\|u\|^{\frac{n\theta}{s}}-q \int_\Om f(u^+)u^+~dx - \int_\Om f^\prime(u^+)(u^+)^2~dx<0\]
which suggests that $u\in \Gamma \setminus \{0\}$, where we use the fact that $\mu \in (0,\mu_0)$. Now by Remark \ref{KSE-remark1}, we deduce that there exists an $\epsilon>0$ and a differentiable function $\tilde \xi : B(0,\epsilon)\subset X_0\to \mb R$ such that
\[\tilde \xi(0)=1\;\; \text{and}\;\; \tilde \xi(v)(u-v)\in \mc N_\mu,\;\; \forall \; v \in B(0,\epsilon).\]
We use the fact that $\phi_{\tilde \xi(0)(u-0)}^{\prime\prime}(1)<0$ and the continuity of maps $\tilde \xi $ and $\phi_{\tilde \xi(v)(u-v)}^{\prime\prime}$  to conclude that there exists a $\epsilon>0$(chosen even smaller, if necessary) such that
\[\tilde \xi(v)(u-v)\in \mc N_\mu^-,\;\; \forall \; v \in B(0,\epsilon).\]
This completes the proof.
\end{proof}
 Now we define th following
 \[\Upsilon_\mu^- = \inf_{v\in \mc N_\mu^-}I_\mu(v).\]

\begin{theorem}\label{KSE-second-sol}
$(P_\mu)$ admits a non negative weak solution $v_0 \in \mc N_\mu^-$ when $\mu \in (0,\mu_0)$ such that $I_\mu(v_0)= \Upsilon^-_\mu$.
\end{theorem}
\begin{proof}
We begin with proving that $\mc N_\mu^-$ is a closed set. To prove this, it is enough to show that there exists a $C_0>0$ such that $\|v\|\geq C_0$ for all $v \in \mc N_\mu^-$. Suppose not then there exists a sequence $\{v_k\}\subset \mc N_\mu^-$ such that $\|v_k\| \to 0$ as $k \to \infty$. But $\phi_{v_k}^{\prime\prime}(1)<0$ gives that there exists constants $C_1,C_2>0$ such that
\[\left(\frac{n\theta}{s}+q-1\right)\|v_k\|^{\frac{n\theta}{s}}\leq C_1 \|v_k\|^r + C_2 \|v_k\|^{r+\beta}\]
which gives a contradiction since $r>\frac{n\theta}{s}$. Therefore we can apply the Ekeland Variational principle to assert that there exists a sequence $\{v_k\}  \subset \mc N_\mu^-$ such that
\begin{align*}
     \Upsilon_\mu^- &\leq I_\mu(v_k)\leq \Upsilon^-_\mu +\frac{1}{k}\\
     I_\mu(v) &\geq I_\mu(v_k)-\frac{1}{k}\|v-v_k\|,\; \text{for all}\; v\in \mc N_\mu^- \cup \{0\}.
 \end{align*}
 So $\{v_k\}$ is a minimizing sequence for $I_\mu$ at $\Upsilon_\mu^-$. As in Lemma \ref{KSE-lemma10}, we obtain that $\{v_k\} \subset \Gamma \setminus \{0\}$. Following proof of Lemma \ref{KSE-lemma6}, it is easy to verify that $\{v_k\}$ is bounded in $X_0$. Therefore, there exists a $v_0 \in X_0\setminus{\{0\}}$ such that
 \[v_k \rightharpoonup v_0\;\text{weakly in} \; X_0 \text{ and}\; v_k \to v_0 \;\text{strongly in} \; L^r(\Om)\; \text{ as}\; k \to \infty \]
 for any $r \geq 1$. Now using Lemma \ref{KSE-lemma10} and realising that Lemma \ref{KSE-lemma9} holds for $\{v_k\}$ replaced with $\{u_k\}$, we follow the proof of Theorem \ref{KSE-first-sol} to obtain that $v\geq 0$ a.e. in $\mb R^n$, $v_0 \in \mc N_\mu^-$ and it solves $(P_\mu)$ weakly such that $\Upsilon_\mu^-= I_\mu(v_0)$. This finishes the proof.
\end{proof}

\noi \textbf{Proof of Theorem \ref{KSE-MT}: } The proof of this main result follows as a consequence of Theorem
\ref{KSE-first-sol} and Theorem \ref{KSE-second-sol}.

\end{document}